\documentclass[a4paper,10pt,leqno]{article}
\usepackage{amsmath,amsfonts,amssymb,amsthm}
\usepackage{graphicx,subfigure,booktabs,multirow}
\usepackage{mathrsfs,enumerate,color,bm}%,mathdots
\usepackage{charter}
\usepackage{geometry}
\newtheorem{theorem}{Theorem}[section]
\newtheorem{corollary}{Corollary}

\newtheorem{proposition}{Proposition}

\theoremstyle{definition}
\newtheorem{definition}[theorem]{Definition}
\newtheorem{algorithm}[theorem]{Algorithm}

\title{A NEW CLASS OF POSITIVE SEMI-DEFINITE TENSORS \thanks{To appear in: Journal of Industrial and Management Optimization}}
\author{
Yi Xu\thanks{Mathematics Department, Southeast University, 2 Sipailou, Nanjing, Jiangsu Province 210096, China. Email: {\tt yi.xu1983@hotmail.com}. This author's work is supported by National Natural Science Foundation of China Nos. 11501100, 11571178 and 11671082.}
\and
Jinjie Liu\thanks{Department of Applied Mathematics, The Hong Kong Polytechnic University, Kowloon, Hong Kong. Email: {\tt jinjie.liu@connect.polyu.hk}.}
\and
Liqun Qi\thanks{Department of Applied Mathematics, The Hong Kong Polytechnic University, Kowloon, Hong Kong. Email: {\tt liqun.qi@polyu.edu.hk}. This author's work is supported in part by the Hong Kong Research Grant Council Nos. PolyU  15302114, 15300715, 15301716 and 15300717.}
}
\date{\today}

\begin{document}
\maketitle

\begin{abstract}
  In this paper, a new class of positive semi-definite tensors, the MO tensor, is introduced. It is inspired by the structure of Moler matrix, a class of test matrices. Then we focus on two special cases in the MO-tensors: Sup-MO tensor and essential MO tensor. They are proved to be positive definite tensors. Especially, the  smallest H-eigenvalue of a Sup-MO tensor is positive and tends to zero as the dimension tends to infinity, and an essential MO tensor is also a completely positive tensor.



\end{abstract}

\bigskip

\section{Introduction}

In recent decades, tensors, as the natural extension of matrices, have been more and more ubiquitous in a wide variety of applications, such as data analysis and mining, signal processing, computational biology and so on \cite{CZPA, KB}. The positive (semi-)definiteness of tensors is an intrinsic property for tensors, which is closely related with the nonnegative polynomials. In 2005, the concepts of positive (semi-)definiteness and eigenvalues of matrices were extended to tensors \cite{Q}. Positive (semi-)definite tensors have widely applications in polynomial optimization, spectral hypergraph theory, magnetic resonance imaging and so on \cite{Q, Q14, QYW}. Though the verification of positive (semi-)definite tensors has been shown to be NP-hard by Hillar and Lim \cite{HL}, the positive (semi-)definiteness of some tensors with special structures have been verified, such as even order symmetric diagonally dominated tensors, even order symmetric B tensor, even order symmetric $M$-tensors and so on, see \cite{LQ16, Q, QL17, QS, ZQZ}.

Furthermore, completely positive tensors, which is connected with nonnegative tensor factorization, also have significant applications in polynomial optimization problems, statistics, data analysis and so on. They were first introduced in \cite{QXX14}. An even order completely positive tensor is a positive semi-definite tensor. In \cite{QL}, two well-known classes of test matrices, Pascal matrices and Lehmer matrices were extended to Pascal tensors and Lehmer tensors. They are easily checkable and were proved to be completely positive tensors \cite{QL}.

There is another class of test matrices, the Moler matrices. A Moler matrix is a positive definite symmetric matrix.  One of its eigenvalues is quite small, and it is usually used for testing eigenvalue computations. In the following section, we prove that the smallest eigenvalue of a Moler matrix approaches 0, when the dimension of the Moler matrix, $n\rightarrow+\infty$. Inspired by the good properties of Moler matrix, we construct a new class of positive semi-definite tensors. Since there have been $M$-tensors already, we name this new class of tensor as MO tensor, which come from first two letters of Moler.

The remaining part of this paper is organized as follows. In Section 2, we review some basic concepts and preliminary results to help our work. The main results are given in Section 3.  In that section, we construct MO tensors, which include Sup-MO tensors and essential MO tensors with the concept of MO value.  Meanwhile, some related properties are proved.   Some final remarks are made in Section 4.

\section{Preliminaries}
In this part, we will review some basic definitions and lemmas which are useful in our following research.

We call $\mathcal{A}(n,m)$ a real $m$th order $n$-dimensional tensor if
$$\mathcal{A}(n,m)=(a_{i_1 \cdots i_m}), i_j\in[n], j\in[m],$$ where $[n] := \{1, \cdots, n\}$.
 All the real $m$th order $n$-dimensional tensors form a $n^{m}$-dimensional linear space denoted by $T_{n, m}.$ Furthermore, if all the entries $a_{i_1 \cdots i_m}$ of $\mathcal{A}(n,m)$ are invariant under any index permutation, $\mathcal{A}(n,m)$ is a symmetric tensor. Denote the set of all the real symmetric tensors of order $m$ and dimension $n$ by $S_{n, m}$. Obviously, $S_{n, m}$ is a linear subspace of $T_{n, m}.$

Let $\mathcal{A}(n,m)=(a_{i_1\cdots i_m})\in T_{n, m}$ and ${\bf x}\in\mathbb{R}^n.$ Then $\mathcal{A}(n,m){\bf x}^m$ is an $m$th degree homogeneous polynomial defined by
$$\mathcal{A}(n,m) {\bf x}^m=\sum_{i_1, i_2, \cdots, i_m=1}^{n} a_{i_1 \cdots i_m}x_{i_1}x_{i_2} \cdots x_{i_m}.$$
We call tensor $\mathcal{A}(n,m)$ as a positive definite ({\bf PD}) tensor if $\mathcal{A}(n,m){\bf x}^m > 0,$ for any non-zero vector ${\bf x}\in\mathbb{R}^n,$ and as positive semi-definite ({\bf PSD}) tensor if $\mathcal{A}(n,m){\bf x}^m\geq 0,$ for any vector ${\bf x}\in\mathbb{R}^n.$ Let ${\bf x} =(x_1, x_2, \cdots, x_n)^{\top}\in\mathbb{R}^n,$ then ${\bf x}^{[m]}$ is a vector in $\mathbb{R}^n$ denoted as
$${\bf x}^{[m]}=(x_1^m, x_2^m, \cdots, x_n^m)^\top .$$

To identify the positive (semi-)definiteness of tensors, the spectral theory of tensors plays an important role for the desired identification. In \cite{Q}, H-eigenvalues and Z-eigenvalues of tensors were introduced. It was shown there that an even order symmetric tensor is positive (semi-)definite if and only if all of its H-eigenvalues or Z-eigenvalues are positive (nonnegative).  Various easily checkable positive (semi-)definite tensors have been discovered consequently \cite{CLQ, LWZZL, QS, ZQZ}.

\begin{definition} \cite{Q}
Let $\mathcal{A}(n,m)\in T_{n, m}, \lambda\in\mathbb{R}.$ If $\lambda$ and a nonzero vector ${\bf x}\in\mathbb{R}^n$ are the solutions of the following polynomial equation:
$$\mathcal{A}(n,m) {\bf x}^{m-1}=\lambda {\bf x}^{[m-1]},$$
then we call $\lambda$ an {\bf H-eigenvalue} of $\mathcal{A}(n,m),$ and ${\bf x}$ an {\bf H-eigenvector} of $\mathcal{A}(n,m)$ associated with the H-eigenvalue $\lambda.$
\end{definition}

\begin{theorem}\label{theo-1} \cite{Q}
Let $\mathcal{A}(n,m)\in S_{n, m},$ and $m$ be even. Then $\mathcal{A}(n,m)$ is positive definite (positive semi-definite) if and only if all of the H-eigenvalues of $\mathcal{A}(n,m)$ are positive (non-negative). Furthermore, we have

(1)
$$\lambda_{min}(\mathcal{A}(n,m))=\min\frac{\mathcal{A}(n,m){\bf x}^m}{\|{\bf x}\|_m^m},$$

(2)
$$\lambda_{min}(\mathcal{A}(n,m))=\min \{\mathcal{A}(n,m){\bf x}^m:\|{\bf x}\|_m =1\},$$

where ${\bf x}\in\mathbb{R}^n,$ and $\|{\bf x}\|_m =(\sum_{i=1}^{n}|x_i|^m)^{\frac{1}{m}}.$
\end{theorem}

\begin{definition} \cite{QL}
Let $\mathcal{A}(n,m)\in S_{n, m}.$ We call $\mathcal{A}(n,m)$ a {\bf completely positive tensor} if there exist an integer $r$ and some ${\bf u}^{(k)}\in\mathbb{R}_{+}^n, k\in [r]$ such that $$\mathcal{A}(n,m)=\sum_{k=1}^{r}({\bf u}^{(k)})^m.$$
\end{definition}

\begin{theorem}\cite{QL}
Let $\mathcal{A}(n,m)\in S_{n, m}.$ If $\mathcal{A}(n,m)$ is a completely positive tensor, then all the H-eigenvalues of $\mathcal{A}$ are nonnegative.
\end{theorem}

\begin{definition} \cite{JN}
Let $\mathcal{A}(n,2)\in \mathbb{R}^{n\times n}$. We call $\mathcal{A}(n,2)$  the $n$-dimensional {\bf Moler matrix} if

$$\mathcal{A}(n,2)_{i,j}=\left\{
             \begin{array}{cc}
               i, & i=j  \\
               \min\{i,j\}-2, & i\neq j \\
             \end{array}
           \right..
$$
\end{definition}

In the following proposition, we give a proof to show that the Moler matrix is a positive definite matrix, and its smallest eigenvalue tends to zero in decreasing as its dimension tends to infinity.

\begin{proposition}
Let $\mathcal{A}(n,2)\in \mathbb{R}^{n\times n}$ be an $n$-dimensional  Moler matrix. (1) $\mathcal{A}(n,2)$ is positive definite; (2) Let $\lambda_{min}(\mathcal{A}(n,2))$ be the smallest eigenvalue of $\mathcal{A}(n,2).$ Then $\lambda_{min}(\mathcal{A}(n,2))\searrow 0.$
\end{proposition}

\begin{proof}

(1) We note that $\mathcal{A} (n,2)=LL^\top$,
where

$$L_{i,j}=\left\{
      \begin{array}{cc}
        1, & i=j \\
        -1, & i>j \\
        0, & i<j \\
      \end{array}
    \right..
$$

Therefore, $\mathcal{A}(n,2)$ is positive definite.

(2) $0 < \lambda_{min}(\mathcal{A}(n+1,2))\leq\lambda_{min}(\mathcal{A}(n,2))$ are easily obtained. Assume that $ {\bf x}=\left(1,\frac{1}{2},\frac{1}{2^2},\cdots,\frac{1}{2^{n-1}}\right)^\top.$ We have $\displaystyle L^\top {\bf x} =\left(                                                                                                                                                        \begin{array}{c}
\displaystyle\frac{1}{2^{n-1}}  \\
\vdots\\
\displaystyle\frac{1}{2^{n-1}}  \\
\end{array}                                                                                                                                                     \right),$  $$\frac{{\bf x}^\top \mathcal{A}(n,2){\bf x}}{{\bf x}^\top {\bf x}}=\frac{3n}{4^n-1}\geq \lambda_{min}(\mathcal{A}(n,2)),$$ which means $\lambda_{min}(\mathcal{A}(n,2))\searrow 0,$ when $n\rightarrow +\infty.$
\end{proof}

\section{Main Results}

In order to introduce a new class of tensors with positive semi-definiteness, we first introduce concepts called the {\bf MO value}, the {\bf MO set} and the {\bf Sup-MO value}.

\begin{definition}
(1) Let $m$ be an even number. We call $\alpha(m)$ as the {\bf MO value}, if $\mathcal{A}(n,m) := \mathcal{M}(n,m)-\alpha(m)\mathcal{N}(n,m)$ is positive semi-definite for any $n$, where
\begin{equation}\label{1}
\mathcal{M}(n,m)_{i_1,i_2,\cdots,i_m}=\left\{
\begin{array}{cc}
i_1,&i_1=i_2=\cdots=i_m\\
    \min\{i_1,i_2,\cdots,i_m\}, & else \\
 \end{array}
 \right.,
\end{equation}

\begin{equation}\label{2}
\mathcal{N}(n,m)_{i_1,i_2,\cdots,i_m}=\left\{
 \begin{array}{cc}
  0, & i_1=i_2=\cdots=i_m \\
  1, & else \\
  \end{array}
   \right..
\end{equation}
   We call the set of all MO values as the {\bf MO set} $\Omega(m);$
(2) We call $\alpha^*(m)=\sup\{\Omega(m)\}$ as the {\bf Sup-MO value}.  We also define {\bf Sub-MO value} $\alpha_*(m)$ of $\Omega(m)$ as $\alpha_*(m)=\inf\{\Omega(m)\}$.
\end{definition}

It is worth noting that $\alpha(m)$ is a parameter only related to $m$. Hence, when we consider to explore its properties, it is necessary to show these properties still hold when $n\rightarrow\infty$.

In this paper, we mainly focus on the properties of $\alpha^*(m)$.
Based on the aforementioned concepts,  MO tensors and Sup-MO tensors are given as following.

\begin{definition}
Let $m$ be an even number. We call $\mathcal{A}(n,m)\in S_{n,m}$ a {\bf MO tensor}, if $$\mathcal{A}(n,m)=\mathcal{M}(n,m)-\alpha(m)\mathcal{N}(n,m),$$
where $\mathcal{M}(n,m)$ and $\mathcal{N}(n,m)$ are defined in $(1)$ and $(2)$, respectively, and $\alpha(m)\in\Omega(m).$

(2) We call $\mathcal{A}(n,m)\in S_{n,m}$ a {\bf Sup-MO tensor}, if $$\mathcal{A}(n,m)=\mathcal{M}(n,m)-\alpha^*(m)\mathcal{N}(n,m).$$

\end{definition}

The theorem in the following shows that the Sup-MO tensor $\mathcal{A}(n, m)$ can be reduced to the Moler matrix when $m = 2$.

\begin{theorem}\label{p-1}
Let $\Omega(m)$ be the MO set, $\mathcal{A}(n,m)\in S_{n,m}$ be a Sup-MO tensor.
When $m=2$, we have $\alpha^*(2)=2=\max\{\Omega(2)\}.$
\end{theorem}

\begin{proof}
 From the property of the Moler matrix, we get $2 \in \Omega(2)$. Then we need to prove that 2 is the Sup-MO value in this case. If $\alpha^*(2)>2,$ then there exists an $\alpha\in (2,\alpha^*(2))\cap\Omega(2)$ such that $$\mathcal{M}(n,2)-\alpha\mathcal{N}(n,2)=\mathcal{M}(n,2)-2\mathcal{N}(n,2)-(\alpha-2)\mathcal{N}(n,2),$$
  where $\mathcal{M}(n,m)$ and $\mathcal{N}(n,m)$ are defined in $(1)$ and $(2)$, respectively.

 Let ${\bf x}=\left(1,\displaystyle\frac{1}{2},\cdots,\displaystyle\frac{1}{2^{n-1}}\right)^\top \in\mathbb{R}^n$, $${\bf x}^\top (\mathcal{M}(n,2)-\alpha\mathcal{N}(n,2)){\bf x}={\bf x}^\top (\mathcal{M}(n,2)-2\mathcal{N}(n,2)){\bf x}-(\alpha-2){\bf x}^\top \mathcal{N}(n,2){\bf x}.$$  Since
$${\bf x}^\top (\mathcal{M}(n,2)-2\mathcal{N}(n,2)){\bf x}=\displaystyle\frac{3n}{4^n-1}, \ {\rm and}\ {\bf x}^\top \mathcal{N}(n,2){\bf x}= \displaystyle\frac{8}{3}-\frac{4}{2^{n-1}}+\frac{4}{3(4^{n-1})},$$ when $n\rightarrow +\infty$, we have $${\bf x}^\top (\mathcal{M}(n,2)-\alpha\mathcal{N}(n,2)){\bf x}<0,$$ which is against to  $\alpha\in \Omega(2).$ Hence $\alpha^*(2)=2.$
\end{proof}

In the following work, a special class of MO tensors, the essential MO tensors are discussed. The following theorem shows the relationship between the essential MO tensors and the completely positive tensors. Then some properties of Sup-MO values in MO tensors are given.

\begin{definition}
Let $\mathcal{A}(n,m)\in S_{n,m}$. We call $\mathcal{A}(n,m)$ the $m$th order $n$-dimensional {\bf essential MO tensor}  if

$$\mathcal{A}(n,m)_{i_1,\cdots,i_m}=\left\{
             \begin{array}{cc}
               i_1, & i_1=i_2=\cdots=i_m  \\
               \min\{i_1,i_2,\cdots,i_m\}-1, & otherwise \\
             \end{array}
           \right..
$$
\end{definition}

\begin{theorem}\label{thm3.5}
Let $\mathcal{A}(n,m)$ be an $m$th order $n$-dimensional essential MO tensor. It is positive definite for all $n$, even $m.$ Furthermore, it is a completely positive tensor for all $n$ and $m$, such as
$\mathcal{A}(n,m)=\sum\limits_{i=1}^n {\bf e}^m_i+\sum\limits_{i=2}^n {\bf r}^m_i$, where $({\bf e}_i)_j=\left\{
\begin{array}{cc}
   1, & j=i \\
   0, & otherwise \\
    \end{array}
  \right.,
$ $({\bf r}_i)_j=\left\{
            \begin{array}{cc}
              1, & j\geq i \\
              0, & otherwise \\
 \end{array}
  \right.$.
\end{theorem}
\begin{proof}
 Let $\mathcal{B}(n,m)_{i_1,\cdots,i_m}=\left\{
                               \begin{array}{cc}
                                 1, & i_1=i_2=\cdots=i_m=1 \\
                                 1, &  i_1,i_2,\cdots,i_m\geq 2\\
                                 0, & otherwise \\
                               \end{array}
                             \right.
$.

Since $\mathcal{B}(n,m)=(1,0,\cdots,0)^m+(0,1,\cdots,1)^m,$ $\mathcal{B}(n,m)$ is a complete positive tensor.
Let $\mathcal{C}(n,m)=\mathcal{A}(n,m)-\mathcal{B}(n,m).$ Then $$\mathcal{C}(n,m)=\left\{
\begin{array}{cc}
  i_1-1, & i_1=i_2=\cdots=i_m\geq 2 \\
      0, & \min\{i_1,i_2,\cdots,i_m\}=1 \\
          \min\{i_1,i_2,\cdots,i_m\}-2, & otherwise   \\
            \end{array}
             \right..
$$
Let $\mathcal{A}(n-1,m)_{i_1,i_2,\cdots,i_m}=\mathcal{C}(n,m)_{i_1+1,i_2+1,\cdots,i_m+1}, i_j\in [n-1],j\in[m].$ Then $\mathcal{A}(n-1,m)$ is an $m$th order $n-1$ dimensional essential MO tensor. Furthermore, if $\mathcal{A}(n-1,m)$ is the completely positive tensor, $\mathcal{A}(n,m)$ is the completely positive tensor.

By the same way, we could get $\mathcal{A}(i,m), i\in [n],$ are all essential MO tensors.
When $n=1$, $\mathcal{A}(1,m)$ is equal to the positive number 1. By induction, we get that $\mathcal{A}(n,m)$ is a completely positive tensor and also a positive definite tensor.
\end{proof}

\begin{corollary}\label{thm3.6}
\begin{flushleft}
(1) Let $\Omega(m)$ be the MO set, and $\mathcal{M}(n,m), \mathcal{N}(n,m)$ be defined in $(1)$ and $(2)$, respectively.
For all $n$ and even $m$, $1$ is always a MO value, i.e., $1\in\Omega(m)$.

(2) $\mathcal{M}(n,m)-\alpha\mathcal{N}(n,m)$ is completely positive for all $n$ and $m$, while  $\alpha\in [0, 1].$
\end{flushleft}
\end{corollary}

\begin{proof}
 (1) Since $\mathcal{M}(n,m)-\mathcal{N}(n,m)$ is an essential MO tensor for all $n$ and even $m$, we get that $1$ is a MO value.

 (2) Since $\mathcal{N}(n,m)={\bf e}^m-\sum\limits_{i=1}^n {\bf e}^m_i$, $$\mathcal{M}(n,m)=\sum\limits_{i=1}^n {\bf e}^m_i+\sum\limits_{i=2}^n {\bf r}^m_i+{\bf e}^m-\sum\limits_{i=1}^n {\bf e}^m_i=\sum\limits_{i=2}^n {\bf r}^m_i+{\bf e}^m,$$
where ${\bf e}=(1,1,\cdots,1)^{\top}.$
So if $\alpha \in [0, 1]$, $$\mathcal{M}(n,m)-\alpha\mathcal{N}(n,m)=\sum\limits_{i=2}^n {\bf r}^m_i+(1-\alpha){\bf e}^m+\alpha\sum\limits_{i=1}^n {\bf e}^m_i$$ is completely positive.
\end{proof}

\begin{corollary}\label{thm3.65}
$\alpha_*(m)$ exists, and $-\frac{1}{2}\leq\alpha_*(m)\leq0.$
\end{corollary}
\begin{proof}
From Corollary \ref{thm3.6}(2), $\alpha_*(m)\leq 0$. Let ${\bf x}=(1,-1,0,\cdots,0)^{\top},$ when $\alpha<-\frac{1}{2}$, $$(\mathcal{M}(n,m)-\alpha\mathcal{N}(n,m)){\bf x}^m=1+2\alpha<0,$$ so $\alpha_*(m)=\inf\{\Omega(m)\}$ exists, and $\alpha_*(m) \geq -\frac{1}{2}.$
\end{proof}

Based on Theorem \ref{thm3.5} and Corollary \ref{thm3.6}, the MO set $\Omega(m)$ is nonempty.
\begin{proposition}\label{pro-alp}
Let $\Omega(m)$ be the MO set, and $\mathcal{M}(n,m), \mathcal{N}(n,m)$ be defined in (\ref{1}) and (\ref{2}), respectively. Then,
\begin{flushleft}
(1) for any $\alpha_1(m), \alpha_2(m)\in \Omega(m),$ $[\alpha_1(m),\alpha_2(m)]\subseteq \Omega(m);$\\
(2) for all even $m\geq 4,$ $1<\alpha^*(m)<2;$\\
(3) $\alpha^*(m)=\max\{\Omega(m)\};$\\
(4) $\alpha^*(m)\searrow 1,$ when $m\rightarrow+\infty.$
\end{flushleft}
\end{proposition}

\begin{proof}
(1) It is obvious.

(2) Since $1\in \Omega(m)$ for all even $m$,  $\Omega(m)\neq \emptyset.$ Let us consider  the tensor $\mathcal{M}(n,m)-2\mathcal{N}(n,m)$ and  $\displaystyle {\bf x}=\left(1,\displaystyle\frac{1}{2},\cdots,\displaystyle\frac{1}{2^{n-1}}\right)^\top \in\mathbb{R}^n.$  Then
$$(\mathcal{M}(n,m)-2\mathcal{N}(n,m)){\bf x}^m=2\displaystyle\sum_{i=1}^n ({\bf e}_i^\top {\bf x})^m+\sum_{i=2}^n({\bf r}_i^\top {\bf x})^m-({\bf e}^\top {\bf x})^m,$$ where ${\bf e}=(1,\cdots,1)^\top $. When $m\geq 4$ and $n\geq 2,$ we have $$\displaystyle\sum_{i=1}^n({\bf e}_i^\top {\bf x})^m=\sum_{i=1}^n\frac{1}{(2^m)^{i-1}}\leq\displaystyle \frac{2^m}{2^m-1}\leq\frac{16}{15},$$
$$\displaystyle\sum_{i=2}^n({\bf r}_i^\top {\bf x})^m\leq\sum_{i=1}^{n-1}\frac{1}{(2^m)^{i}}\leq\displaystyle \frac{2^m}{2^m-1}\leq \frac{16}{15},$$ and $$({\bf e}^\top {\bf x})^m=(\sum_{i=1}^n\displaystyle\frac{1}{2^{i-1}})^m=\displaystyle\left(\frac{1-\displaystyle\frac{1}{2^{n}}}{1-\displaystyle\frac{1}{2}}\right)^m\geq\left(\frac{3}{2}\right)^m.$$  Then $$(\mathcal{M}(n,m)-2\mathcal{N}(n,m)){\bf x}^m\leq \frac{48}{15}-\left(\frac{3}{2}\right)^m<0.$$ Therefore, for all even $m$, $2$ is an upper bound of $\Omega(m)$.   Then $\alpha^{*}(m)$ exists and $\alpha^*(m)<2$.

Now we prove $\alpha^*(m)>1$.
Let
$$\mathcal{U}(n,m)=\mathcal{M}(n,m)-\mathcal{N}(n,m), \mathcal{V}(n,m;\beta)=\mathcal{U}(n,m)-\beta\mathcal{N}(n,m), \beta=\alpha-1.$$ First, we need to prove that there exists $\beta\in(0,1)$ such that $$\mathcal{V}(n,m; \beta){\bf x}^m=(1+\beta)\sum_{i=1}^n({\bf e}_i^\top {\bf x})^m+\sum_{i=2}^n({\bf r}_i^\top {\bf x})^m-\beta({\bf e}^\top {\bf x})^m\geq 0,$$ for all $n\in\mathbb{R}$ and ${\bf x}\in\mathbb{R}^n$ satisfying $||{\bf x}||_m=1.$

If ${\bf e}^\top {\bf x}=0$, then $\mathcal{V}(n,m; \beta){\bf x}^m\geq 0$  for all $\beta \in [0,1]$. If ${\bf e}^\top {\bf x}\neq 0$, then there exists ${\bf y}=(y_1, y_2, \cdots, y_n)\in\mathbb{R}^n$ and ${\bf z}=(z_1, z_2, \cdots, z_n)\in\mathbb{R}^n$ such that
$$y_1={\bf e}^\top {\bf x}, y_i={\bf r}_i^\top {\bf x}, i=2,\cdots,n,$$ and
$$z_i=\displaystyle\frac{y_i}{y_1}, i=1,\cdots,n.$$
Then $z_1=1$ and $$\mathcal{V}(n,m; \beta){\bf x}^m=y_1^m\left[(1+\beta)\left(\sum_{i=1}^{n-1}(z_i-z_{i+1})^m+z_n^m\right)+\sum_{i=2}^n z_i^m-\beta\right].$$
Let
\begin{equation}
g_{n,m}({\bf z},\beta)=(1+\beta)\left(\sum_{i=1}^{n-1}(z_i-z_{i+1})^m+z_n^m\right)+\sum_{i=2}^n z_i^m,
\end{equation}
and
\begin{equation}\label{eq-g}
f_{n,m}(\beta)~=\displaystyle\min_{{\bf z}\in\mathbb{R}^n, z_1=1} ~g_{n,m}({\bf z},\beta).
\end{equation}
It is easy to get $$0\leq f_{n+1,m}(\beta)\leq f_{n,m}(\beta), \mbox{for all}~ \beta \in [0,1].$$ Hence,
\begin{equation}
f_{m}(\beta)=\displaystyle\lim_{n\rightarrow+\infty}f_{n,m}(\beta)
\end{equation}
exists for all $\beta\in[0,1]$.

 In fact, if there exists $ \beta \in [0,1]$ such that $f_m(\beta)=0,$ then by the definition of $f_{n,m}(\beta),$ there exists ${\bf z}^*\in\mathbb{R}^n$ with $z_1^*=1$ such that $(1- z_2^*)^m<\varepsilon$ and $z_2^*<\varepsilon,$ for any $\varepsilon>0,$ which is impossible.  Then $f_m(\beta)>0$. Furthermore, $f_m(0)>0.$

Additionally, when $m\geq 4,$ assuming that ${\bf z}^*=\left(1, \displaystyle\frac{1}{2},0,\cdots,0\right)^\top,$ we get $g_{n,m}({\bf z}^*,1)=\displaystyle\frac{5}{2^m}.$ Thus $f_m(1)\leq f_{n,m}(1)\leq \displaystyle\frac{5}{2^m}<1.$

Moreover, since $m$ is even, $$g_{n,m}({\bf z},\beta_1)\leq g_{n,m}({\bf z},\beta_2), \mbox{for all}~ {\bf z}\in\mathbb{R}^n, \beta_1, \beta_2 ~\mbox{with}~ 0\leq \beta_1\leq\beta_2.$$
Then $f_{n,m}(\beta_1)\leq f_{n,m}(\beta_2)$ and $f_m(\beta_1)\leq f_m(\beta_2),$ which means that $f_m(\beta)$ is a nondecreasing function in $\beta$ on $[0,1].$

Then we prove that $f_m(\beta)$ is a continuous function in $\beta$ on $(0,1)$. Denote $f_m(\beta^+)$, $f_m(\beta^-)$, as the right-hand and left-hand limit on $\beta \in (0,1)$,  $f_m(0^+),$ $f_m(1^-)$  as the right-hand limit on $0$, and left-hand limit on $1$ of the function $f_m(\beta)$, respectively. Since $f_m(\beta)$ is a nondecreasing function, for any $\beta\in(0,1),$ $f_m(\beta^+)$, $f_m(\beta^-)$, $f_m(0^+),$ $f_m(1^-)$ exist and $f_m(\beta^+)\geq f_m(\beta^-)$. Assuming that $f_m(\beta^+)>f_m(\beta^-)$, and $\delta=\frac{f_m(\beta^+)-f_m(\beta^-)}{2}>0,$ for $0<\beta_1<\beta<\beta_2$, there exists $N^*$, when $n>N^*$, we have
\begin{equation*}
\begin{aligned}
f_{n,m}(\beta_1)&\leq f_{m}(\beta_1)+\displaystyle\frac{f_m(\beta^+)-f_m(\beta^-)}{2}\\
&\leq f_{m}(\beta^-)+\displaystyle\frac{f_m(\beta^+)-f_m(\beta^-)}{2}\\
&=\displaystyle\frac{f_m(\beta^+)+f_m(\beta^-)}{2},\\
\end{aligned}
\end{equation*}
and
\begin{equation*}
f_{n,m}(\beta_2)\geq f_m(\beta_2)\geq f_m(\beta^+).
\end{equation*}
Thus, when $n>N^*$, $$f_{n,m}(\beta_2)-f_{n,m}(\beta_1)\geq \displaystyle\frac{f_m(\beta^+)-f_m(\beta^-)}{2}=\delta>0.$$

Since $g_{n,m}({\bf z},\beta)$ is continuous in $\beta$, and the level set of $g_{n,m}({\bf z},\beta)$ is bounded, according to Proposition 4.4 in \cite{BS}, $f_{n,m}(\beta)$ is continuous in $\beta$. When $\beta_1\rightarrow\beta, \beta_2\rightarrow\beta$, we obtain that $f_{n,m}(\beta^+)-f_{n,m}(\beta^-)\geq \displaystyle\frac{f_m(\beta^+)-f_m(\beta^-)}{2}=\delta>0$, which is a contradiction. Hence, $f_m(\beta)$ is continuous in $\beta$ on $(0,1)$.

Finally, $$\partial(f_{n,m}(\beta))_{\beta}=\left\{\sum_{i=1}^{n-1}(z_i^*-z_{i+1}^*)^m+({z_n^*})^m\right\},$$ where ${\bf z}^*\in\displaystyle\arg\min_{{\bf z}\in\mathbb{R}^n,z_1=1}g_{n,m}({\bf z},\beta)$. Noting that $f_{n,m}(\beta)< 1,~ \mbox{for any}~ \beta\in [0,1]$, there exists $N>0$ such that, when $n>N$,  $||\eta_{n,m}||< 1, \forall \eta_{n,m} \in \partial(f_{n,m}(\beta))_{\beta}$ and $||\eta_m||< 1, \forall \eta_m\in \partial(f_{m}(\beta))_{\beta}$.

Hence, for all ${\bf z}\in\mathbb{R}^n$ with $z_1=1$, there exists only one $\beta^*(m)\in(0,1)$ satisfying $f_m(\beta^*(m))=\beta^*(m).$ Besides that, when $0<\beta\leq{\beta}^*(m),$ $g_{n,m}({\bf z},\beta)\geq f_{n,m}(\beta)\geq f_m(\beta)\geq\beta;$ When ${\beta}^*(m)\leq\beta<1, f_m(\beta)\leq\beta.$
This means that there exists $1>\beta>0$ satisfying $$\mathcal{V}(n,m; \beta){\bf x}^m\geq 0,$$ and there exists $\alpha=1+\beta>1$, satisfying $$\mathcal{M}(n,m)-\alpha\mathcal{N}(n,m)\succeq 0.$$ Thus $\alpha^*(m)>1.$

In addition, ${\beta}^*(m)=\alpha^*(m)-1$ is proved. Apparently, $\alpha^*(m)-1\geq \beta^*(m).$ If $\alpha^*(m)-1> \beta^*(m)$, there exists
$\beta_1$ such that $\beta^*(m)<\beta_1<\alpha^*(m)-1.$ Since $f_m(\beta_1)<\beta_1$, by the definition of $f_m(\beta),$ there exists $N>0$ such that, when $n>N,$ $f_{n,m}(\beta_1)-\beta_1<0.$ Therefore, $\mathcal{U}(n,m)-\beta_1\mathcal{N}(n,m)\not\succeq 0,$ which is a contradiction. Thus $\beta^*(m)=\alpha^*(m)-1.$

 (3) From Theorem \ref{p-1}, when $m=2,$ $\alpha^*(2)=\max\{\Omega(2)\}.$ According to (2), when $m\geq 4,$ $\mathcal{M}(n,m)-\alpha^*(m)\mathcal{N}(n,m)\succeq 0,$ which means that $\alpha^*(m)\in \Omega(m).$ Therefore, $\alpha^*(m)=\max\{\Omega(m)\},$ for all even $m$.

(4) Since $0<f_{n,m}(\beta)<1, \beta\in [0,1],$ when $m$ is even and $m\geq4,$ there exists a ${\bf z}^*\in\arg\displaystyle\min_{{\bf z }\in \mathbb{R}^n,z_1=1} g_{n,m}({\bf z},\beta)$ satisfying
$$(z_i^*-z_{i+1}^*)^m\leq 1, i=1,\cdots,n-1, (z_i^*)^m\leq 1, i=2,\cdots,n.$$ Then
\begin{equation*}
\begin{aligned}
&(z_i^*-z_{i+1}^*)^{m+2}\leq (z_i^*-z_{i+1}^*)^m, i=1,\cdots,n-1;\\
&(z_i^*)^{m+2}\leq (z_i^*)^m, i=2,\cdots,n.
\end{aligned}
\end{equation*}
Hence, $f_{n,m+2}(\beta)\leq f_{n,m}(\beta),$ and there exists $N>0$, when $n>N,$ $f_{m+2}(\beta)\leq f_{m}(\beta).$ By the definition of ${\beta}^*(m)$ above, without losing generality, assuming that $0\leq\beta\leq{\beta}^*(m),$ we obtain that
 $f_{m-2}(\beta)\geq f_m(\beta)\geq \beta.$ Therefore, ${\beta}^*(m-2) \geq {\beta}^*(m)$, which means $\alpha^*(m-2)\geq \alpha^*(m).$

 Since $\beta^*(m)\geq 0$ and $\beta^*(m)$ is nonincreasing, $\beta^*=\displaystyle\lim_{m\rightarrow+\infty}\beta^*(m)$ exists.
If $\beta^*\neq 0,$ then there exists a $N>0,$ when $n>N$ and $m>N,$ we have $f_{n,m}(\beta^*(m))\geq f_m(\beta^*(m))=\beta^*(m)\geq \beta^*>0.$ However, in the other side, when $n>N$ and $m>N,$ $f_{n,m}(\beta^*(m))\rightarrow 0$, which is a contradiction. Hence, $\beta^*=0$ and $\displaystyle\lim_{m\rightarrow+\infty}\alpha^*(m)=1.$
\end{proof}

We could use the following algorithm to compute $\alpha^*(m)$.
\begin{algorithm}[Computing $\alpha^*(m)$]\label{alg-1}
\begin{itemize}
  \item[]
  \item[S1:] Let $\beta^0=1, \beta^1=\beta_0^1=1, n=1, k=0, \varepsilon>0,$ $m$ be a even number ;
  \item[S2:] Solve the problem (\ref{eq-g}) to get $f_{n,m}(\beta_k^n);$
  \item[S3:] If $f_{n,m}(\beta_k^n)-\beta_k^n<-\varepsilon$, then $\beta_{k+1}^n=\frac{\beta_k^n}{2}$, if $f_{n,m}(\beta_k^n)-\beta_k^n>\varepsilon$, $\beta_{k+1}=\frac{\beta_k^n+1}{2}$, $k=k+1$, goto S2; else denote $\beta^n=\beta^n_k$
  \item[S4:] If $|\beta^n-\beta^{n-1}|<\varepsilon$, stop and output $\beta^{n}+1$; else $n=n+1$, $k=0$, goto S2.
\end{itemize}
\end{algorithm}
It is not hard to compute  $f_{n,m}(\beta)$  since (\ref{eq-g}) is a convex problem.
By computing, we get that  $\alpha^*(4)=1.1429$, $\alpha^*(6)=1.0323$, $\alpha^*(8)=1.0079$.
In the final work, a property of the minimal eigenvalue of the Sup-MO tensor is obtained. Then, the positive definiteness of
Sup-MO tensors is proved.

\begin{theorem}\label{p-3}
Let $\Omega(m)$ be an MO set. For all even $m\geq 4$ and $\alpha^*(m)=\max\{\Omega(m)\},$ $\mathcal{A}(n,m; \alpha^*(m))$ is a Sup-MO tensor, i.e.,
\begin{equation*}
\begin{aligned}
&\mathcal{A}_{i_1,\cdots,i_m}(n,m; \alpha^*(m))=\mathcal{M}(n,m)-\alpha^*(m)\mathcal{N}(n,m)\\
&=\left\{
             \begin{array}{cc}
               i_1, & i_1=i_2=\cdots=i_m,  \\
               \min\{i_1,i_2,\cdots,i_m\}-\alpha^*(m), & otherwise. \\
             \end{array}
           \right.
\end{aligned}
\end{equation*}
 Assume that $\lambda_{min}(\mathcal{A}(n,m; \alpha^*(m)))$ is the smallest eigenvalue of $\mathcal{A}(n,m; \alpha^*(m)).$ Then, \\$\lambda_{min}(\mathcal{A}(n,m; \alpha^*(m)))$ strictly decreases to $0$, when $n\rightarrow\infty$. Furthermore, $\mathcal{A}(n,m; \alpha^*(m))$ is positive definite.
\end{theorem}

\begin{proof} By Theorem \ref{theo-1}, it is easy to see that $\lambda_{min}(\mathcal{A}(n,m; \alpha^*(m)))$ decreases in $n$, for all even $m$. In the following, we prove that it is strictly decreasing to $0$. $g_{n,m}({\bf z},\beta), f_{n,m}(\beta)$ and $f_{m}(\beta)$ are defined as (3-5).

Since $1>f_m(\beta^*(m))=\beta^*(m)>0$, $f_{n,m}(\beta^*(m))\rightarrow\beta^*(m)$. Suppose that ${\bf z}^*\in\displaystyle\arg\min_{{\bf z}\in\mathbb{R}^n, z_1=1} g_{n,m}({\bf z},\beta^*(m))$. Because $z_1^*=1$, $||{\bf z}^*||_m\geq 1$. Let $w_i^*=z_i^*-z_{i+1}^*,i=1,\cdots,n-1$, $w_n^*=z_n^*.$ Then, when $m\geq 4,$
  $$\beta^*(m)\leq g_{n,m}({\bf z}^*,\beta^*(m))=f_{n,m}(\beta^*)\leq f_{n,m}(1)<\frac{5}{2^m}<1.$$

By the definition of $g_{n,m}({\bf z}^*,\beta^*(m)),$ $\displaystyle\sum_{i=2}^n (z_i^*)^m\leq
  \frac{5}{2^m}<1$, which means that $|z_2^*|\leq {\displaystyle\frac{5^\frac{1}{m}}{2}}<1.$ Thus $w_1^*=1-z_2^*\geq 1-\displaystyle\frac{5^\frac{1}{m}}{2}$. Hence, when $m\geq 4,$ $||{\bf w}^*||_m\geq 1-\displaystyle\frac{5^\frac{1}{m}}{2}.$

 According to Theorem \ref{theo-1} $$0\leq\lambda_{min}(\mathcal{A}(n,m; \alpha^*(m)))\leq \displaystyle\frac{\mathcal{A}(n,m; \alpha^*(m))({\bf w}^*)^m}{||{\bf w}^*||_m^m}.$$ By the definition of ${\bf w}^*$,
   $\mathcal{A}(n,m; \alpha^*(m))({\bf w}^*)^m\rightarrow 0,$ when $n\rightarrow+\infty$. Since $||{\bf w}^*||_m\geq 1-\displaystyle\frac{5^\frac{1}{m}}{2},$ we get that $\lambda_{min}(\mathcal{A}(n,m; \alpha^*(m)))\rightarrow 0$, when $n\rightarrow\infty$.

   Then we prove that the decreasing of $\lambda_{min}(\mathcal{A}(n,m; \alpha^*(m)))$ is strict. Consider the following program:
   \begin{eqnarray}
      \min & \mathcal{A}(n,m; \alpha^*(m)){\bf x}^m \nonumber\\
      s.t.& ||{\bf x}||_m=1 .\nonumber
   \end{eqnarray}
   Then its KKT conditions are
\begin{equation}
 \begin{cases}
   & \mathcal{A}(n,m; \alpha^*(m)){\bf x}^{m-1}=\lambda{\bf x}^{[m-1]} \\
    & ||{\bf x}||_m=1.
\end{cases}
\end{equation}

   The smallest solution $\lambda_{n,m}$ and the corresponding vector ${\bf x}\in\mathbb{R}^n$  of above program are the smallest H-eigenvalue and H-eigenvector of $\mathcal{A}(n,m; \alpha^*(m))$. If $\lambda_{n,m}= \lambda_{n+1,m}$ for some $n$,
   then there exist ${\bf x}\in\mathbb{R}^n$ and $\bar{{\bf x}}\in\mathbb{R}^{n+1}$ with $\bar{{\bf x}}=({\bf x}^\top,0)^\top$ satisfying
 $$\mathcal{A}(n,m; \alpha^*(m)){\bf x}^{m-1}=\lambda_{n,m}{\bf x}^{[m-1]},$$
 $$\mathcal{A}(n+1,m; \alpha^*(m))\bar{{\bf x}}^{m-1}=\lambda_{n+1,m}\bar{{\bf x}}^{[m-1]}.$$ Hence,
$$\sum_{i_2,\cdots,i_m=1}^{n+1}\mathcal{A}(n+1,m; \alpha^*(m))_{n+1,i_2,\cdots,i_m} \bar{x}_{i_2} \cdots \bar{x}_{i_m}=\lambda_{n+1,m}\bar{x}_{n+1}^{m-1}.$$
    Since $\bar{{\bf x}}=({\bf x}^\top,0)^\top,$  the above equation is
$$\sum_{i_2,\cdots,i_m=1}^{n}\mathcal{A}(n+1,m; \alpha^*(m))_{n+1,i_2,\cdots,i_m} x_{i_2} \cdots x_{i_m}=0.$$
     Because
$$\sum_{i_2,\cdots,i_m=1}^{n}\mathcal{A}(n,m; \alpha^*(m))_{n,i_2,\cdots,i_m} x_{i_2} \cdots x_{i_m}=\lambda_{n,m}x_{n}^{m-1},$$
we have
\begin{equation*}
\begin{aligned}
&\sum_{i_2,\cdots,i_m=1}^{n}\mathcal{A}(n,m; \alpha^*(m))_{n,i_2,\cdots,i_m} x_{i_2} \cdots x_{i_m}\\
-&\sum_{i_2,\cdots,i_m=1}^{n}\mathcal{A}(n+1,m; \alpha^*(m))_{n+1,i_2,\cdots,i_m} x_{i_2} \cdots x_{i_m} \\
=&\alpha^*(m)x_{n}^{m-1}=\lambda_{n,m}x_{n}^{m-1}.
\end{aligned}
\end{equation*}

According to the above proof, $\alpha^*(m)>1>\lambda_{n,m}$. Therefore,  $x_n=0$. By the same discussion, we get ${\bf x}=0$, which is against $||{\bf x}||_m=1.$ Thus, $\lambda_{min}(\mathcal{A}(n,m; \alpha^*(m)))$  strictly decreases.
Finally, together with Corollary \ref{thm3.6}, we have $\mathcal{A}(n,m; \alpha^*(m))$ is positive definite.
\end{proof}

\section{Final Remarks}

In this paper, we construct the MO tensor with introducing the concepts of the MO value and the MO set. Then we mainly discuss two special cases of the MO tensor: the Sup-MO tensor and the essential MO tensor.  We prove that an even order essential MO tensor is  a completely positive tensor and positive definite. Then, some related properties of the Sup-MO value of an even order Sup-MO tensor are given. Furthermore, the positive definiteness of an even order Sup-MO tensor is proved, since the minimal H-eigenvalue of the Sup-MO tensor strictly decreases to 0, when $n\rightarrow\infty$. In the future work, some of the applications of the Sup-MO tensor and the properties of Sub-MO value and Sub-MO tensor will be discussed.

\medskip

There are three further research questions for the MO tensor, the Sup-MO tensor and  the Sub-MO tensor.

1. Are the Sup-MO tensors  SOS (sum-of-squares) tensors?  For the definition of SOS tensors, see \cite{CLQ}.    An SOS tensor is a PSD tensor, but not vice versa.    This theory can be traced back to David Hilbert \cite{Hi}.    We randomly tested some Sup-MO tensors, and found that they were SOS tensors.  We will study this in the future study.

2. We cannot verify whether $\alpha_*(m)$ can be reached or not. Hence, we do not know whether $\Omega(m)$ is compact or not. If the $\Omega(m)$ is compact, how to get the length of the MO set $\Omega(m)$ is also an interesting work.  We will continue to explore the properties of $\alpha_*(m)$ and $\Omega(m)$.

3. Since the properties of the Moler matrices make them to be good test matrices for the linear equations and eigensystems, we will try to find that if the Sup-MO tensor can also be a good candidate for testing in some tensor computation software packages or not.

%\bibliographystyle{abbrv}
%\nocite*
%\bibliography{bib_emten}

%For acknowledgements section, please don't number the section, please begin it with \section*{Acknowledgements}

% You may incorporate your references as follows in your main tex file.
% Using BibTex is not recommended but can be handled.

\end{document}